\documentclass{amsproc}

\usepackage[T2A]{fontenc} 
\usepackage[cp1251]{inputenc}
\usepackage[english]{babel} 
\usepackage{amsthm}
\usepackage{amsmath}
\usepackage{amssymb}
\usepackage{amsfonts}
\usepackage{mathtext}
\usepackage{mathrsfs}
\usepackage{cite}

\usepackage{a4wide}
\usepackage{amsmath}
\usepackage{enumerate}

\usepackage{amsmath,amsthm}
\usepackage{amsfonts}
\usepackage{amssymb}
\usepackage{graphicx, color, soul}
\usepackage[T1]{fontenc}
\usepackage{verbatim}

\usepackage{amscd}

\theoremstyle{plain}
\newtheorem{theorem}{\indent Theorem}

\newtheorem*{theorem*}{\indent Theorem} 
\newtheorem*{lemma*}{\indent Lemma}
\newtheorem*{corollary*}{\indent Corollary}
\newtheorem*{prop*}{\indent Proposition}
\newtheorem*{propos*}{\indent Proposition}
\newtheorem*{hypothesis*}{\indent Hypothesis}
\theoremstyle{definition}

\newtheorem{example}{\indent Example}
\newtheorem*{definition*}{\indent Definition}
\newtheorem*{remark*}{\indent Remark}
\newtheorem*{example*}{\indent Example}

\renewenvironment{proof}{\indent\textbf{Proof. }}{\hfill$\Box$}

\newcommand{\Hol}{\,\mathrm{Hol}\,}

\begin{document}
	
	\author{E.~A.~Kokin}\thanks{$^1$Saratov National Research State University named after N. G. Chernyshevsky, Russia, 410012, Saratov, Astrakhanskaya st., 83, evgeny@epromicro.com}
	
	\title{Isomorphism for the Holonomy Group \\ of a K-Contact Sub-Riemannian Space}
	
	\begin{abstract}
		The holonomy group of the adapted connection on a K-contact Riemannian manifold $(M, \theta, g)$ is considered. It is proved that if the orbit space $M/\xi$ of the Reeb field $\xi$ action admits a manifold structure, then the holonomy group of the adapted connection on $M$ is isomorphic to the holonomy group of the Levi-Civita connection on the Riemannian manifold $(M/\xi, h)$, where $h$ is the induced Riemannian metric on $M/\xi$. Thanks to this result, a simple proof of the de Rham theorem for the case of K-contact sub-Riemannian manifolds is obtained, stating that if the holonomy group of the adapted connection on $M$ is not irreducible, then the orbit space $M/\xi$ is locally a product of Riemannian manifolds.
		
		
	\end{abstract}
	

	\maketitle

	\section*{Introduction}
	
	Nonholonomic, contact, and sub-Riemannian manifolds occupy an important place in modern differential geometry. These spaces have numerous applications both in geometry and mathematical physics, geometric control theory, and analysis on manifolds with constraints given by nonintegrable distributions. Recently, particular attention has been paid to the study of connections naturally adapted to the sub-Riemannian and contact structure, see for example the monographs \cite{Kokin_Agr19,Kokin_Blair,Kokin_Gal1}. In particular, the adapted connection allows the study of global and local geometric properties of contact-type sub-Riemannian manifolds, such as curvature characteristics, Ricci and Einstein tensors, as well as questions related to holonomy groups. 
	
	We describe the content and results of the present article. In Section \ref{Kokin_sec1} the basics of holonomy group theory of Riemannian manifolds and the de Rham theorem are given. In Section \ref{Kokin_sec2} K-contact sub-Riemannian manifolds, the adapted connection and its holonomy group are defined. Section \ref{Kokin_sec3} gives a version of de Rham's theorem for K-contact sub-Riemannian manifolds recently proved in \cite{Kokin_Gro}. Section \ref{Kokin_sec4} contains the central result of this paper, describing the relationship between the holonomy group of the adapted connection of a K-contact sub-Riemannian manifold and the holonomy group of the Levi-Civita connection of the induced metric on the orbit space of the Reeb field action under the assumption that the orbit space admits a natural smooth manifold structure. An example of the contactization of a K\"ahler manifold is considered. The obtained isomorphism of holonomy groups allowed to give in Section \ref{Kokin_sec5} a simple proof of de Rham's theorem for K-contact sub-Riemannian manifolds. An example of an $S^1$-bundle over a symmetric K\"ahler space is considered.
	
	\section{Holonomy Groups and de Rham's Theorem for Riemannian Manifolds}\label{Kokin_sec1}

	This section presents results from \cite{Kokin_Besse}.
	Let $(N,h)$ be a connected Riemannian manifold. Denote by $\nabla^h$ the Levi-Civita connection on $N$. The components of the Levi-Civita connection are defined by the equality
	\begin{equation}\label{Kok_GammaLC}
	\Gamma^{hk}_{ij}=\frac{1}{2}g^{kl}\left(\partial_{x^i}g_{lj}+\partial_{x^j}g_{li}-\partial_{x^l}g_{ij}  \right).
	\end{equation}
	Each smooth curve $\gamma:[a,b]\to N$ defines the parallel transport
	$$\tau^h_\gamma:T_{\gamma(a)}N\to  T_{\gamma(b)}N,$$ defined as follows. 
	A vector field $X(t)$ along the curve $\gamma(t)$ is called parallel if
	$$\nabla_{\dot\gamma(t)}X(t)=0.$$
	This equation can be written in local coordinates as
	$$\dot X^i+\Gamma^{hi}_{jk}X^j\dot\gamma^k=0,$$ where $$X(t)=X^i(t)\partial_{x^i}.$$
	This shows that for each vector $X_a \in T_{\gamma(a)}N$ there exists a unique parallel vector field $X(t)$ along the curve $\gamma(t)$ such that
	$X(a)=X_a$, then the value $X(b)$ is the result of parallel transporting the vector $X_a$ along the curve $\gamma(t)$. The notion of parallel transport naturally extends to piecewise-smooth curves.
	
	Fix a point $p \in N$. The holonomy group $\Hol_p(\nabla^h)$ of the Levi-Civita connection of the Riemannian manifold at $p\in N$ is defined as the group of parallel transports along all piecewise-smooth loops at $p$. The holonomy group is a Lie subgroup of the orthogonal Lie group $\mathrm{O}(T_pN,h_p)$. Because the manifold $N$ is connected, the holonomy groups at all points of the manifold are isomorphic.

	Suppose that $(N,h)$ is a product of Riemannian manifolds
	$$(B_1,g_1)\times\cdots\times (B_r,g_r).$$
	Then the holonomy group $\Hol_p(\nabla^h)$ is the product of the corresponding holonomy groups:
	$$\Hol_p(\nabla^h)=\Hol_{p_1}(\nabla^{g_1})\times\cdots\times \Hol_{p_r}(\nabla^{g_r}),$$
	where $$p=(p_1,\dots,p_r)\in B_1\times\cdots\times B_r=N.$$ This means $\Hol_p(\nabla^h)$ preserves the decomposition
	$$T_pN=T_{p_1}B_1\oplus\cdots\oplus T_{p_r}B_r,$$ and hence the subgroup $\Hol_p(\nabla^h)\subset \mathrm{O}(T_pN,h_p)$ is not irreducible.
	
	Conversely, suppose $\Hol_p(\nabla^h)$ is not irreducible. Since the subgroup $\Hol_p(\nabla^h)\subset \mathrm{O}(T_pN,h_p)$ is completely reducible (the orthogonal complement to each invariant subspace is also invariant), there is a $\Hol_p(\nabla^h)$-invariant orthogonal decomposition
	\begin{equation}\label{Kok_eqRozlHolN}T_pN=F_1\oplus\cdots\oplus F_r.\end{equation}
	We formulate the local version of the de Rham theorem for Riemannian manifolds.
	
	\begin{theorem}\label{Kok_ThdeR}
		Suppose the holonomy group $\Hol_p(\nabla^h)$ of the Levi-Civita connection of the Riemannian manifold $(N,h)$ preserves the decomposition \eqref{Kok_eqRozlHolN}.
		Then for each point of $N$ there exists an open neighborhood $U\subset N$ such that the manifold $(U,h|_{U})$ is isometric to the product of Riemannian manifolds
		$$(B_1,g_1)\times\cdots\times (B_r,g_r).$$
	\end{theorem}

	\section{Holonomy Groups of K-Contact Sub-Riemannian Manifolds}\label{Kokin_sec2}

	A contact sub-Riemannian manifold is a triple $(M,\theta,g)$ where $\theta$ is a contact form on a smooth manifold $M$, and $g$ is a sub-Riemannian metric on the contact distribution $$D=\ker\theta.$$ Denote by $\xi$ the corresponding Reeb vector field uniquely defined by two conditions:
	$$\theta(\xi)=1,\quad \iota_\xi (d\theta)=0.$$
	The metric $g$ defines a horizontal connection (an analogue of the Levi-Civita connection)
	$$\nabla^g:\Gamma(D)\times \Gamma(D)\to\Gamma(D).$$
	
	We will consider adapted coordinate systems \cite{Kokin_Gal1,Kokin_Gal3,Kokin_Gal4}. Suppose the dimension of $M$ is $n=2m+1$. A coordinate system $x^1,\dots,x^n$ on $M$ is called adapted if  
	$$\xi=\partial_{x^n}.$$ Let indices $i,j,k,l$ take values $1,\dots, 2m$. If $y^1,\dots,y^n$ is another adapted coordinate system, then on the intersection of charts the relation
	$$y^i=y^i(x^1,\dots,x^{2m}), \quad y^n=x^n+F(x^1,\dots,x^{2m})$$
	holds.
	Since $$\theta(\xi)=1,$$ then relative to adapted coordinates we can write
	$$\theta=dx^n+\Gamma^n_idx^i.$$
	Therefore vector fields
	$$e_i=\partial_{x^i}- \Gamma^n_i\partial_{x^n}$$
	define a basis of the layers of the contact distribution $D$ in points of the domain of adapted coordinates.
	The coefficients of the horizontal connection $\nabla^g$ can be found by the formula
	\begin{equation}\label{Kok_Gammag}
	\Gamma^{gk}_{ij}=\frac{1}{2}g^{kl}\left(e_ig_{lj}+e_jg_{li}-e_lg_{ij}  \right),
	\end{equation}
	where
	$$g_{ij}=g(e_i,e_j).$$
	
	A contact sub-Riemannian manifold is called K-contact if
	$$L_\xi g=0.$$
	This means that in each adapted system of coordinates we have
	\begin{equation}\label{Kok_Kcont}
	\partial_{x^n}g_{ij}=0.
	\end{equation} We will henceforth consider only K-contact sub-Riemannian manifolds.
	
	The adapted connection on a K-contact sub-Riemannian manifold \cite{Kokin_FGR} is the connection $$\nabla^A:\Gamma(TM)\times \Gamma(D)\to\Gamma(D)$$ in the vector bundle $D\to M$ defined by
	$$\nabla^A_XY=\nabla^g_XY,\quad \nabla^A_\xi Y=[\xi,Y],\quad X,Y\in\Gamma(D).$$
	Its components are
	$$\Gamma^{Ak}_{ij}=\Gamma^{gk}_{ij}, \quad \Gamma^{Ak}_{nj}=0.$$
	The adapted connection is a special $N$-connection \cite{Kokin_Gal1,Kokin_Gal3,Kokin_Gal4}.
	
	The connection $\nabla^A$ defines parallel transport of tangent vectors to the contact distribution $D$ along piecewise-smooth curves on $M$. The holonomy group of the connection $\nabla^A$ at a point $x\in M$ is denoted by $\Hol_x(\nabla^A)$. These holonomy groups were first considered in \cite{Kokin_FGR}. There is an inclusion $\Hol_x(\nabla^A)\subset\mathrm{O}(D_x,g_x)$. Holonomy groups of adapted connections at different points are isomorphic.
	In recent work \cite{Kokin_GalHolK} it was shown that $\Hol_x(\nabla^A)$ is the holonomy group of some (abstract) Riemannian manifold, however the connection between this Riemannian manifold and the considered K-contact metric manifold was not established. 
	
	\section{De Rham Theorem for K-Contact Sub-Riemannian Manifolds}\label{Kokin_sec3}

	If the holonomy group $\Hol_x(\nabla^A)\subset\mathrm{O}(D_x,g_x)$ of the adapted connection is not irreducible, then by complete reducibility $\Hol_x(\nabla^A)$ preserves an orthogonal decomposition
	\begin{equation}\label{Kok_eqRozlHolM}T_xM=E_1\oplus\cdots\oplus E_r.\end{equation}
	
	The Reeb vector field $\xi$ defines a pseudogroup of diffeomorphisms of the manifold $M$. Denote by $M/\xi$ the orbit space of this pseudogroup. In general, the space $M/\xi$ has a complicated structure and is not a smooth manifold. However, each point of $M$ has an open neighborhood $U$ such that $U/\xi$ admits a natural manifold structure and there is a surjective submersion $$\mathrm{pr}:U\to U/\xi,$$ whose differential restricted to the layers of the contact distribution
	$$d_x\mathrm{pr}:D_xU\to T_{\mathrm{pr}(x)}U/\xi$$
	is an isomorphism of vector spaces.
	In this case, since the manifold is K-contact, the sub-Riemannian metric $g|_U$ defines a Riemannian metric on $U/\xi$ which we denote by $h_U$.

	The main result of \cite{Kokin_Gro} is the following theorem generalizing the local de Rham theorem for Riemannian manifolds.
	
	\begin{theorem}\label{Kok_ThGro}
		Let $(M,D,g)$ be a K-contact oriented sub-Riemannian manifold. Assume that the holonomy group of the adapted connection preserves the decomposition \eqref{Kok_eqRozlHolM}.
		Then for each point of $M$ there is an open neighborhood $U\subset M$ such that the manifold $(U/\xi,h_U)$ is isometric to the product of Riemannian manifolds$$(B_1,g_1)\times\cdots\times (B_r,g_r).$$
	\end{theorem}
	
	\section{Isomorphism of Holonomy Groups}\label{Kokin_sec4}

	In this section we prove the following theorem.
	\begin{theorem}\label{Kok_ThHolIsom}Let $(M,D,g)$ be a K-contact sub-Riemannian manifold. Suppose the orbit space $M/\xi$ admits a smooth manifold structure such that the projection
		$$\mathrm{pr}:M\to M/\xi$$
		is a surjective submersion. Denote by $h$ the induced Riemannian metric on the orbit space $M/\xi$. Then for each point $x\in M$ there is an isomorphism of holonomy groups
		$$\Hol_x(\nabla^A)\cong\Hol_{\mathrm{pr}(x)}(\nabla^h).$$
	\end{theorem}
	
	\begin{proof} Recall that $\dim M=n=2m+1$. Consider an arbitrary adapted coordinate system $x^1,\dots,x^n$ on $M$. Then the functions $x^1,\dots,x^{2m}$ can be considered as local coordinates on the orbit space $M/\xi$. Relative to these coordinates the projection $$\mathrm{pr}:M\to M/\xi$$ looks like $$(x^1,\dots,x^{2m},x^n)\mapsto (x^1,\dots,x^{2m}).$$ By definition of the metric $h$ the equality $$h=g_{ij}dx^idx^j$$ holds. Since the manifold $(M,\theta,g)$ is K-contact, equality \eqref{Kok_Kcont} holds. From \eqref{Kok_GammaLC} and \eqref{Kok_Gammag} it follows that $$\Gamma^{gk}_{ij}=\Gamma^{Ak}_{ij}.$$ Let $\gamma:[a,b]\to M$ be a curve on the manifold $M$. In arbitrary adapted coordinates $x^1,\dots,x^n$ on $M$ we have $$\gamma(t)=\big(x^1(t),\dots,x^{2m}(t),x^n(t)\big).$$ Then the curve $\mathrm{pr}\circ\gamma$ on $M/\xi$ is given by $$(\mathrm{pr}\circ\gamma)(t)=\big(x^1(t),\dots,x^{2m}(t)\big).$$ From the parallel transport equations it follows that a section $$X(t)=X^i e_i$$ is parallel along $\gamma$ if and only if the vector field $$X^i\partial_{x^i}$$ is parallel along the curve $\mathrm{pr}\circ\gamma$. This shows that parallel transports along curves $\gamma$ and $\mathrm{pr}\circ\gamma$ determine each other. Moreover, the commutative diagram holds
		$$\begin{CD}
		D_{\gamma(a)} @>\tau^A_{\gamma}>> D_{\gamma(b)} \\
		@Vd_{\gamma(a)}\mathrm{pr}VV    @VVd_{\gamma(b)}\mathrm{pr}V \\
		T_{\mathrm{pr}(\gamma(a))}M/\xi @>\tau^h_{\mathrm{pr}\circ\gamma}>> T_{\mathrm{pr}(\gamma(b))}M/\xi
		\end{CD}$$

		Fix a point $x\in M$. Let $\gamma:[a,b]\to M$ be a piecewise smooth loop at $x\in M$. Then $\mathrm{pr}\circ\gamma$  is a piecewise smooth loop at $p=\mathrm{pr}(x)$ and the relation
		$$\tau^h_{\gamma}=d_{x}\mathrm{pr}\circ \tau^A_{\gamma}\circ (d_{x}\mathrm{pr})^{-1}$$
		holds. This defines a homomorphism
		$$\psi_x:\Hol_x(\nabla^A)\to\Hol_{\mathrm{pr}(x)}(\nabla^h)$$
		of Lie groups. Clearly, this homomorphism is injective. We prove that $\psi_x$ is surjective. Let $\mu:[a,b]\to N$ be a piecewise smooth loop at $p=\mathrm{pr}(x)$. Due to compactness of $\mu([a,b])$, the image $\mu([a,b])$ can be covered by a finite number of coordinate neighborhoods $V_1,\dots,V_r$, which are images of coordinate neighborhoods $U_1,\dots,U_r$ on $M$ with adapted coordinates. It can also be assumed that only neighboring charts have nonempty intersections. Fix numbers $$c_1,\dots,c_{r+1}\in [a,b],$$ such that
		$$c_1=a,\quad \mu(c_2)\in V_1\cap V_2,\dots,\mu(c_r)\in V_{r-1}\cap V_r,\quad c_{r+1}=b.$$
		Define a curve $\gamma:[a,b]\to M$ covering $\mu$. Let $$\big(x^1(t),\dots,x^{2m}(t)\big)$$ be the coordinate expression of the segment $\gamma|_{[c_s,c_{s+1}]}$. Define the curve $\gamma_i:[c_s,c_{s+1}]\to U_{s}\subset M$ in coordinates by $$\big(x^1(t),\dots,x^{2m}(t),f(t)\big).$$ Obviously, the functions $f(t)$ can be chosen so that the curves $\gamma_i$ together define a loop $\gamma$ at $x\in M$. Clearly, the loop $\gamma$ covers the loop $\mu$. This proves that the homomorphism $\psi_x$ is an isomorphism. The theorem is proved. \end{proof}
	
	\begin{example} 
		Recall that a Sasaki manifold is a triple $(M,\xi,g)$ where $(M,g)$ is a Riemannian manifold, endowed with a unit Killing field $\xi$, and the endomorphism field $\Phi$ defined by
		$$\Phi(X)=\nabla^g_X\xi,\quad X\in\Gamma(TM),$$ satisfies the conditions 
		\begin{align*}
		\Phi^2&=-\mathrm{id}_{TM}+\theta\otimes \xi,\\
		(\nabla^g_X\Phi)Y&=g(\xi,Y)X-g(X,Y)\xi,\quad\forall X,Y\in\Gamma(TM),\\
		\end{align*}    
		where $\theta$ is the 1-form dual to the vector field $\xi$. In this case $(M,\theta,g)$ is a K-contact sub-Riemannian manifold, and $\xi$ is the corresponding Reeb vector field.
		
		Let $A$ be either of the Lie groups $S^1$ or $\mathbb{R}$. Recall that the Sasaki manifold $(M,\xi,g)$ is called regular if the Reeb vector field $\xi$ generates a free and proper $A$-action. A regular Sasaki manifold $(M,\xi,g)$ is called the \emph{contactization} of a K?hler manifold $(N,h,\omega)$ if there exists a map $$\mathrm{pr}: M \to N,$$ defining a principal $A$-bundle such that
		$$g=\mathrm{pr}^* h,\quad d\theta=\mathrm{pr}^*\omega,$$
		see, for example, \cite{Kokin_ACHK}. 
		
		From theorem \ref{Kok_ThHolIsom} it follows that if a regular Sasaki manifold $(M,\xi,g)$ is a contactization of a K?hler manifold $(N,h,\omega)$, then there is an isomorphism of holonomy groups $$\Hol_x(\nabla^A)\cong\Hol_{\mathrm{pr}(x)}(\nabla^h).$$
	\end{example}
	
	\section{Proof of de Rham's Theorem for K-Contact Sub-Riemannian Manifolds}\label{Kokin_sec5}

	In this section we give a simple proof of theorem \ref{Kok_ThGro}.
	Suppose the holonomy group $\Hol_x(\nabla^A)$ of the adapted connection preserves the decomposition \eqref{Kok_eqRozlHolM}. Obviously, for every open subset $W \subset M$ such that $x \in M$ we have an inclusion $$\Hol_x(\nabla^A|_W)\subset \Hol_x(\nabla^A),$$ and consequently the holonomy group $\Hol_x(\nabla^A|_W)$ preserves the decomposition \eqref{Kok_eqRozlHolM}. Let $U$ be an open neighborhood of $x$ such that $U/\xi$ is a manifold. By theorem \ref{Kok_ThHolIsom} the holonomy group $\Hol_x(\nabla^A|_U)$ is isomorphic to the holonomy group $\Hol_p(\nabla^h|_{U/\xi})$, therefore $\Hol_p(\nabla^h|_{U/\xi})$ preserves the decomposition \eqref{Kok_eqRozlHolN}, where $$E_s=(d_{x}\mathrm{pr}) F_s,\quad s=1,\dots,r.$$ By theorem \ref{Kok_eqRozlHolN} there is an open subset $V \subset U/\xi$ such that the Riemannian manifold $(V,h_V)$ is isometric to the product of Riemannian manifolds. By shrinking $U$ if necessary, we can assume $U/\xi=V$. This proves theorem \ref{Kok_ThGro}.

	\begin{example} Consider the product of compact K\"ahler symmetric spaces
		$$(N,h)=(N_1,h_1)\times\cdots\times(N_r,h_r).$$ As above, the holonomy group of $(N,h)$ is the product of the holonomy groups of the factors.
		It is known \cite{Kokin_BFG} that the $S^1$-bundle obtained from the canonical bundle over $(N,h)$ is a sub-symmetric K-contact sub-Riemannian manifold $(M,\theta,g)$. According to theorem \ref{Kok_ThHolIsom} there is an isomorphism  $$\Hol_x(\nabla^A)\cong\Hol_{\mathrm{pr}(x)}(\nabla^h).$$ The de Rham decomposition for $(M,\theta,g)$ coincides with the original decomposition for  $(N,h)$.
	\end{example}


\end{document}